\documentclass[graybox]{svmult}

\usepackage{mathptmx}       
\usepackage{helvet}         
\usepackage{courier}        
\usepackage{type1cm}        
\usepackage{makeidx}         
\usepackage{graphicx}        
\usepackage{multicol}        
\usepackage[bottom]{footmisc}

\usepackage{amssymb}
\usepackage{dsfont} 

\newcommand{\be}{\begin{equation}}
\newcommand{\ee}{\end{equation}}
\newcommand{\bel}[1]{\begin{equation}\label{#1}}
\newcommand{\bea}{\begin{eqnarray}}
\newcommand{\eea}{\end{eqnarray}}
\newcommand{\balign}{\begin{align}}
\newcommand{\ealign}{\end{align}}
\newcommand{\ba}{\begin{array}}
\newcommand{\ea}{\end{array}}
\newcommand{\bfig}{\begin{figure}}
\newcommand{\efig}{\end{figure}}
\newcommand{\eref}[1]{(\ref{#1})}

\newcommand{\ket}[1]{\mbox{$| \, {#1}\, \rangle$}}
\newcommand{\exval}[1]{\mbox{$\langle \, {#1}\, \rangle$}}

\newcommand{\half}{\frac{1}{2}}

\newcommand{\rme}{\mathrm{e}}
\newcommand{\rmd}{\mathrm{d}}

\newcommand{\Z}{{\mathbb Z}}
\newcommand{\N}{{\mathbb N}}
\renewcommand{\S}{{\mathbb S}}

\makeindex             

\begin{document}

\title*{On the Fibonacci universality classes in nonlinear fluctuating
hydrodynamics}
\author{G.M.~Sch\"utz}
\institute{
G.M.~Sch\"utz \at Institute of Complex Systems II,
Forschungszentrum J\"ulich, 52425 J\"ulich, Germany \email{g.schuetz@fz-juelich.de}}
%
%
\maketitle

%
\abstract{We present a lattice gas model that without fine 
tuning of parameters is expected to exhibit the so far elusive
modified Kardar-Parisi-Zhang (KPZ) universality class. To this end, we review briefly how non-linear 
fluctuating hydrodynamics in one dimension 
predicts that all dynamical universality classes in its range of applicability belong to an
infinite discrete family which we call Fibonacci family
since their dynamical exponents are the Kepler ratios $z_i = F_{i+1}/F_{i}$ of neighbouring
Fibonacci numbers$F_i$, including diffusion ($z_2=2$), KPZ ($z_3=3/2$), and
the limiting ratio which is the golden mean $z_\infty=(1+\sqrt{5})/2$. Then we 
revisit the case of two conservation
laws to which the modified KPZ model belongs. We also derive criteria on the
macroscopic currents to lead to other non-KPZ universality
classes.
}

\section{Introduction}

It is well-known that in one dimension transport in stationary states
is usually anomalous even when the microscopic interactions are short-ranged
and noise is uncorrelated \cite{Lepr16}. This property manifests itself
in transport coefficients that diverge, usually algebraically, with system size, in contrast to
normal, i.e., diffusive transport where the transport coefficients are material-dependent
constants.
Also the spatio-temporal fluctuations of the densities $\rho_{\alpha}(x,t)$
of globally conserved quantities $N_\alpha$ (such as mass, energy, etc.) are frequently
characterized by non-diffusive scaling properties with dynamical exponents
$z_\alpha \neq 2$, including the celebrated  Kardar-Parisi-Zhang (KPZ)
universality class where $z=3/2$ \cite{Halp15,Spoh17}, and another universality class
with $z=3/2$ \cite{Bern14,Bern16,Popk14a,Spoh14}. 

If the large-scale behaviour of the fluctuations
is dominated by the long wave length modes of the conserved fields 
then the theory of nonlinear fluctuating hydrodynamics
(NLFH) \cite{Spoh14} predicts that, in a comoving frame with collective velocity $v_\alpha$,
the normalized dynamical structure functions, i.e., the
stationary correlations $S_{\alpha}(x,t):=\exval{\phi_{\alpha}(x,t)\phi_{\alpha}(0,0)}$ 
of the centered normal modes $\phi_{\alpha}(x,t) = \sum_{\beta} R_{\alpha\beta} (\rho_{\beta}(x,t)
-\rho_{\beta})$, have a scaling limit of the form 
$S_{\alpha}(x,t) = t^{-1/z_\alpha} f_{\alpha} ((x-v_\alpha t)/(\lambda_\alpha t)^{z_\alpha})$. 
Here $f_{\alpha}(\cdot)$ is a {\it universal}
scaling function that does not depend depend on the microscopic details of the interaction. Non-universal
are the scale factors $\lambda_\alpha$ as well as the collective velocities and the coefficients
$R_{\alpha\beta}$ that both depend on the stationary densities $\rho_{\alpha}=\exval{\rho_{\alpha}(x,t)}$
and the stationary currents $j^\alpha(\rho_1,\dots\rho_n)$ associated with the conserved quantities.
For diffusion the scaling function is a Gaussian, while for the KPZ universality class one has the
Pr\"ahofer-Spohn function \cite{Prae04}.

Thus the main quantities of interest in the study of spatio-temporal fluctuations
in one space dimension are the dynamical exponents $z_\alpha$ and the scaling functions
$f_{\alpha}(\cdot)$. They determine the dynamical universality class that a given microscopic
model belongs to. Remarkably, it was found \cite{Popk15b} that the possible dynamical exponents are 
either sequences of the Kepler ratios $z_i = F_{i+1}/F_{i}$ of neighbouring Fibonacci numbers $F_i$ 
beginning with $z=2=2/1$ or with $z=3/2$, or the golden mean $z=(1+\sqrt{5})/2$ which is the limiting
value $i\to\infty$ of the Kepler sequence. Also the corresponding scaling functions 
have been determined, with one exception, which is the so-called modified Kardar-Parisi-Zhang 
universality class \cite{Spoh15} for which, however, until now no generic microscopic model has 
been proposed.

In the following we briefly review the reasoning that leads to these predictions. We follow mainly the arguments
put forward in Refs. \cite{Spoh14} and \cite{Popk16} which lead, via mode coupling theory, to the conclusion that the dynamical
universality class of a mode can be deducted from the above-mentioned macroscopic stationary current-density relation
$j^\alpha(\rho_1,\dots\rho_n)$  through the so-called mode-coupling matrices (Sec. 2). Then we revisit the case of 
two conservation laws studied in some detail already in \cite{Popk15a} and \cite{Spoh15}. In Sec. 3 we construct a 
microscopic lattice gas model that, without fine-tuning of parameters,
is predicted to be in the modified Kardar-Parisi-Zhang universality class. Finally, in Sec. 4, 
we present in a ``consumer-friendly''
fashion the criteria on the currents $j^\alpha(\rho_1,\dots\rho_n)$ under which only non-KPZ universality
classes appear in systems with two conservation laws.

\section{Nonlinear fluctuating hydrodynamics}

\subsection{Notation and general properties of fluctuations}

In order to fix ideas we consider discrete microscopic models that evolve in continuous time $t$
and that have $n$ locally conserved quantities.
By this we mean the following. Let $\S$ be some set, $\Lambda$ denote a contiguous set of integers,
and $\eta_k\in \S$ with $k \in \Lambda$ be the local state variable. The index $k$ denotes a lattice site or a 
particle in a chain, depending on the type of model one has in mind. A microscopic configuration
at time $t$ is thus given by $\eta(t) = \{\eta_k(t):k\in\Lambda\}$.\footnote{When the time $t$
is irrelevant we drop the dependence on $t$.} The generator of the dynamics
is denoted by $\mathcal{L}$. The translation operator is denoted by $\mathcal{T}$ and defined
by the property $\mathcal{T}(\eta_k) = \eta_{k+1}$, and similar for functions of the
local state variables. We assume the dynamics to be translation invariant, i.e., 
$\mathcal{T} \mathcal{L} = \mathcal{L} \mathcal{T}$, with the identification $\eta_k \equiv \eta_{k+L}$
if $\Lambda$ is the integer torus.

In order to introduce conservation laws
consider a cylinder function  $\xi^\alpha_0(\eta)$
where $\alpha \in \{1,\dots, n\}$ and define $\xi^\alpha_{k}(\eta) := \mathcal{T}^k(\xi^\alpha_0(\eta))$.
We shall assume that (i) the $\xi^\alpha_k(\eta)$
satisfy the discrete continuity equations
\bel{dce}
\mathcal{L} (\xi^\alpha_k(\eta)) = j^\alpha_{k-1}(\eta) - j^\alpha_k(\eta)
\ee
for all $\alpha$, (ii) only the $\xi^\alpha_k(\eta)$ have this property, and
(iii) that also the so-called microscopic currents $j^\alpha_k(\eta)$ are cylinder functions. 
We shall drop the dependence of the conserved quantities $\xi^\alpha_k$
and currents $j^\alpha_k$ on the configuration $\eta$. 

Following \cite{Gris11} we postulate that there exists a family of 
translation invariant grand-canonical measures parametrized by fugacities $\varphi^\alpha$ which are
translation invariant and invariant under the dynamics generated by $\mathcal{L}$. Expectations under
this measure are denoted by $\exval{\cdot}$. In particular, we introduce the stationary
conserved densities $\rho_\alpha := \exval{\xi^\alpha_k}$ and
the stationary currents $j^\alpha := \exval{j^\alpha_k}$. The first and second
derivatives of the currents $j^\alpha$ w.r.t. the densities $\rho_\beta,\rho_\gamma$ 
are denoted by $j^\alpha_{\beta}$ and $j^\alpha_{\beta\gamma}$. 

The discrete fluctuation fields at time $t$ are the centered random variables 
$\bar{\rho}^\alpha_k(t) := \xi^\alpha_k(t) - \rho_\alpha$. 
The dynamical structure matrix $\mathbf{\bar{S}}_k(t)$ is defined by
the matrix elements
\bel{Sbar}
\bar{S}^{\alpha\beta}_k(t) := \exval{\bar{\rho}^\alpha_k(t)\bar{\rho}^\beta_0(0)}.
\ee
The compressibility matrix $\mathbf{K} = \sum_k\mathbf{\bar{S}}_k(t)$ is the covariance matrix of the 
conserved quantities which is independent
of time due to the conservation laws \eref{dce}. The current Jacobian
$\mathbf{J}$ is the matrix with matrix elements $J_{\alpha\beta}=j^\alpha_{\beta}$ and the
Hessians $\mathbf{H}^\alpha$ have matrix elements $H^\alpha_{\beta\gamma}=j^\alpha_{\beta\gamma}$.
The $n$-dimensional unit matrix is denoted by $\mathds{1}$.
Transposition of a matrix is denoted by the superscript $T$. 

We make the mild (but essential !) assumptions on the invariant measure that for all $\alpha,\beta$
one has $K_{\alpha\beta} < \infty$ and
$\lim_{n\to\infty} n \exval{\bar{\rho}^\alpha_0 j^\beta_n} = 0$. These 
hypotheses imply the Onsager-type current symmetry \cite{Gris11}
\bel{Ons}
\frac{\partial j^\beta}{\partial \phi^\alpha} = \frac{\partial j^\alpha}{\partial \phi^\beta}
\ee
for all $\alpha,\beta$.\footnote{It seems to have gone unnoticed that, quite remarkably, 
this symmetry relates a purely  static property of the invariant measure (the covariances 
$K_{\alpha\beta}$) with the microscopic dynamics which give rise to the currents $j^\alpha$. 
This restricts severely the possible microscopic dynamics for which a given 
measure can be invariant.} In particular, the chain rule implies \cite{Spoh14}
\bel{Ons2}
\mathbf{JK} = (\mathbf{JK})^T.
\ee
The current symmetry \eref{Ons} ensures that the current Jacobian has real eigenvalues which we denote
by $v_\alpha$. Counterexamples with non-decaying correlations are models that exhibit phase separation \cite{Chak16,Kafr02,Rama02}.
It should be noted that the assumption of locality of the
conserved quantity and the associated current as well as the finite number $n$ of
conservation laws also rules out models with infinitely many and non-local
conservation laws. Nevertheless, some of the phenomenology of such models
seems to be similar to the finite and local case \cite{Kund16}.

Throughout this work we shall assume complete absence of degeneracy
of the eigenvalues $v_\alpha$. 
Then one can always write $\mathbf{V} := \mathbf{R}\mathbf{J}\mathbf{R}^{-1} =
\mbox{diag}(v_1,\dots , v_n)$.
By convention we choose the normalization 
\bel{modenorm}
\mathbf{R}\mathbf{K}\mathbf{R}^T = \mathds{1}.
\ee
The eigenmodes are defined to be the transformed
fluctuation fields $\phi^\alpha_k(t) := \sum_\beta R_{\alpha\beta} \bar{\rho}^\beta_k(t)$.
They give rise to the normal form of the dynamical structure matrix
\bel{S}
\mathbf{S}_k(t) := \mathbf{R} \mathbf{\bar{S}}_k(t) \mathbf{R}^T 
\ee
which satisfies $\sum_k \mathbf{S}_k(t) = \mathds{1}$ for all $t$. 
The conservation law, translation invariance and the mild
decay of stationary correlations as assumed for the invariant measure yields the exact
relation $d/dt \sum_k k \mathbf{S}_k(t) = \mathbf{V}$ for all $t$.

We point out that the dynamical structure functions $S^{\alpha\beta}_k(t)$ have an 
alternative meaning as describing the relaxation of a microscopic perturbation of the 
invariant measure at the origin $k=0$ \cite{Popk15a}. As the perturbation evolves in time,
it separates into distinct density peaks, one for each mode $\alpha$. The eigenvalues $v_\alpha$ are 
the center-of-mass velocities of these perturbations. The variance w.r.t. $k$ of the 
diagonal structure function $S^{\alpha\alpha}_k(t)$ describes, on lattice scale, the spatial spreading
of mode $\alpha$.

\subsection{Nonlinear fluctuating hydrodynamics}

In the hydrodynamic limit where the ``lattice'' spacing tends to zero we denote
the scaling limits of the density by $\rho_\alpha(x,t)$,
$\bar{\rho}_\alpha(x,t)$, and $\phi_\alpha(x,t)$ resp.
Under Eulerian scaling one expects from the law of large numbers and local stationarity \cite{Kipn99}
that the discrete continuity equation \eref{dce} gives rise to the system of conservation laws
$\partial_t \rho_\alpha(x,t) + \partial_x j_\alpha(x,t) = 0$
where the currents $j_\alpha(x,t)$ depend on $x$ and $t$ only through the local
densities $\rho_\alpha(x,t)$ via the stationary current-density relation. Thus one can write
$j_\alpha(x,t) = j_\alpha(\rho_1(x,t),\dots,\rho_n(x,t))$ which gives
\bel{shcl}
\partial_t \rho_\alpha(x,t) + \sum_\beta 
J_{\alpha\beta}(\rho_1(x,t),\dots,\rho_n(x,t)) \partial_x \rho_\beta(x,t) = 0.
\ee
Non-degeneracy of $\mathbf{J}$ implies that this nonlinear system of conservation laws
is strictly hyperbolic. 
Obviously, the constant functions $\rho_\alpha(x,t)=\rho_\alpha$ form
a translation invariant stationary solution.

In order to study fluctuations one expands around a fixed stationary solution
and adds to the current a 
phenomenological diffusion term with diffusion matrix $\mathbf{\tilde{D}}(\rho_1,\dots,\rho_n)$ 
and Gaussian white noise $\tilde{\zeta}^\alpha(x,t)$ with an amplitude that is usually taken to
satisfy the fluctuation-dissipation theorem. 
Renormalization group
arguments suggest that only terms up to second order in the density expansion are relevant. Third order
terms may lead to logarithmic corrections to the fluctuations, but only if the
second-order term vanishes. All higher order terms vanish in the scaling limit of large
$x$ and large $t$ \cite{Devi92}. Thus, omitting arguments, one arrives at the non-linear
fluctuating hydrodynamics equation \cite{Spoh14}
\bel{nlfh}
\partial_t \bar{\rho}_\alpha(x,t) +  \partial_x  \left[\sum_\beta 
\hat{J}_{\alpha\beta}\bar{\rho}_\beta(x,t) + \half
\sum_{\beta\gamma} \bar{\rho}_\beta(x,t) H^\alpha_{\beta\gamma} \bar{\rho}_\gamma(x,t) + \tilde{\zeta}_\alpha(x,t)\right] = 0
\ee
with linear current operator $\hat{J}_{\alpha\beta} = J_{\alpha\beta} - \tilde{D}_{\alpha\beta} \partial_x$.

In terms of the eigenmodes one has
\bel{nlfh2}
\partial_t \phi_{\alpha}(x,t) + \partial_x  \left[\sum_{\beta} 
\hat{V}_{\alpha\beta} \partial_x \phi_{\beta}(x,t)
+ \sum_{\beta\gamma}\phi_{\beta}(x,t) G^{\alpha}_{\beta\gamma}  \phi_{\gamma}(x,t) + 
\zeta_{\alpha}(x,t)\right] = 0
\ee
with 
$\hat{V}_{\alpha\beta} = v_\alpha \delta_{\alpha\beta} - D_{\alpha\beta} \partial_x$
where $\mathbf{D} =  \mathbf{R} \tilde{\mathbf{D}} \mathbf{R}^{-1}$,
the symmetric mode coupling matrices
\bel{mcm}
\mathbf{G}^{\alpha} = \half \sum_{\lambda} R_{\alpha \lambda} (\mathbf{R}^{-1})^T \mathbf{H}^\lambda \mathbf{R}^{-1},
\ee
and transformed noise $\zeta_{\alpha}(x,t) = \sum_{\lambda} R_{\alpha\lambda} \tilde{\zeta}_{\lambda}(x,t)$
with covariance $\exval{\zeta_{\alpha}(x,t)\zeta_{\beta}(x',t')} = 2 D_{\alpha\beta} \delta(x-x')\delta(t-t')$.
One recognizes in \eref{nlfh2} a system of coupled noisy Burgers equations which with the substitution
$\bar{\rho}_\alpha(x,t) = \partial_x h_\alpha(x,t)$ turns into a system of coupled KPZ equations 
\cite{Erda93,Ferr13,Funa15,Schu17}.

\subsection{Mode coupling theory}

Following \cite{Spoh14} one writes the stochastic pde \eref{nlfh2} in discretized form 
$\phi_\alpha(x,t) \to \phi_\alpha(n,t)$ with $n\in\Z$ in terms of a generator $L=L_0+L_1$ 
where $L_0$ represents the linear part involving $\hat{\mathbf{V}}$ and $L_1$ represents the
non-linear part involving the mode-coupling matrices $\mathbf{G}^{\alpha}$.
This yields $S_{\alpha\beta}(n,t) = \exval{\phi_\beta(0,0) \rme^{Lt} \phi_\alpha(n,0)}$
and therefore
\bel{discnlfh} 
\frac{\rmd}{\rmd t} S_{\alpha\beta}(n,t) = \exval{\phi_\beta(0,0)  \rme^{Lt} L_0 \phi_\alpha(n,0)} +
\exval{\phi_\beta(0,0)  \left(\rme^{Lt} L_1 \phi_\alpha(n,0)\right)}.
\ee
The discretization of the generator is chosen such that a product of mean-zero Gaussian measures
for the $\phi_\alpha(n)\equiv\phi_\alpha(n,0)$ is invariant under the stochastic evolution.

We insert the identity $\rme^{Lt} = \rme^{L_0 t} + \int_0^t \rmd s \, \rme^{ L_0(t- s)} L_1 \rme^{L s}$ 
into the second term on the r.h.s. of \eref{discnlfh}. The first contribution involving only the
linear evolution vanishes
since by closer inspection one realizes that one is left with the expectation of cubic
terms which are zero. The second contribution involves higher order correlators which due to the Gaussian
measure can be factorized into pair correlations using the Wick rule. Finally, one replaces the bare evolution 
$\rme^{ L_0(t- s)}$ by the interacting evolution $\rme^{ L(t- s)}$ and takes the continuum
limit.  One arrives at the mode coupling equation \cite{Spoh14}
\bea
\partial_t S_{\alpha\beta}(x,t) & = & - v_\alpha \partial_x S_{\alpha\beta}(x,t)
+ \sum_\gamma D_{\alpha\gamma} \partial_x^2 S_{\gamma\beta}(x,t) \nonumber \\
\label{mctfull}
& & + \int_0^t \rmd s \, \int_{-\infty}^\infty \rmd y \, \partial_y^2 \sum_\gamma
M_{\alpha\gamma}(y,s) S_{\gamma\beta}(x-y,t-s)
\eea
with the memory term
\bel{mtfull}
M_{\alpha\gamma}(y,s) = 2 \sum_{\mu\mu'\nu\nu'} G^\alpha_{\mu\nu} G^\gamma_{\mu'\nu'} S_{\mu\mu'}(y,s)S_{\nu\nu'}(y,s).
\ee

Next we recall that in the strictly hyperbolic case all modes drift with different velocities.
Hence after time $t$ their centers of mass are a distance of order $t$ apart. On the other hand, the broadening
of the peaks is expected to grow sublinearly. Hence, eventually the offdiagonal terms $S_{\alpha\beta}(x,t)$ die out and
can be neglected. With $S_{\alpha}(x,t)\equiv S_{\alpha\alpha}(x,t)$ the mode-coupling equations thus reduce to
\bea
\partial_t S_{\alpha}(x,t) & = & - v_\alpha \partial_x S_{\alpha}(x,t)
+ D_{\alpha\alpha} \partial_x^2 S_{\alpha}(x,t) \nonumber \\
\label{mctdiag}
& & + \int_0^t \rmd s \, \int_{-\infty}^\infty \rmd y \, \partial_y^2 
M_{\alpha\alpha}(y,s) S_{\alpha}(x-y,t-s)
\eea
with the memory term
\bel{mtdiag}
M_{\alpha\alpha}(y,s) = 2 \sum_{\mu\nu} \left(G^\alpha_{\mu\nu}\right)^2 S_{\mu}(y,s)S_{\nu}(y,s).
\ee

\subsection{Fibonacci universality classes}

It was found in \cite{Popk15b,Popk16} that the mode coupling equations \eref{mctdiag} have 
scaling solutions which can be obtained in explicit form after Fourier and Laplace transformation.
In order to classify the solutions we recall the recursive definition of the Fibonacci 
numbers $F_{n+1} = F_n + F_{n-1}$ with $F_1=F_2=1$ and introduce the set 
$\mathbb{I}_\alpha := \{\beta: G^{\alpha}_{\beta\beta} \neq 0\}$
of modes $\beta$ that give rise to a non-linear term in the time-evolution of mode $\alpha$.

\begin{theorem}[Refs. \cite{Popk15b,Popk16}] 
\label{Theo:Fibonacci}
Let $u=p t^{1/z_{\alpha}}$ with dynamical exponent
$z_\alpha > 1$. Then\\[2mm]
(1) In Fourier representation 
$\hat{S}_\alpha(p,t) := \frac{1}{\sqrt{2\pi}}\int_{-\infty}^\infty \rmd x\,
\rme^{- i p x} S_\alpha(x,t)$
the mode coupling equation \eref{mctdiag} with memory term
\eref{mtdiag} has for finite $|u|$ the scaling solution
\bel{scalingformFT2}
\lim_{t\to\infty} \rme^{iv_\alpha pt} \hat{S}_\alpha(p,t) =
 \hat{f}_{\alpha} (u)
\ee
where\\[2mm]
$\bullet$ Case 1: For modes $\alpha$ such that $\mathbb{I}_\alpha = \emptyset$
one has $z_{\alpha} =2$ and $ \hat{f}_{\alpha} (u) = \frac{1}{\sqrt{2\pi}} \rme^{- D_\alpha u^2}$ (diffusive universality class).\\[2mm]
$\bullet$ Case 2: For modes $\alpha$ such that  $\mathbb{I}_\alpha \neq \emptyset$ and $\alpha \notin \mathbb{I}_\alpha$ 
the dynamical exponents satisfy the nonlinear recursion
$z_\alpha = \min_{\beta \in \mathbb{I}_\alpha}\left[\left(1 + \frac{1}{z_\beta}\right)\right] $ and the scaling function
is given by
$ \hat{f}_{\alpha} (u) = \frac{1}{\sqrt{2\pi}} \rme^{
- E_\alpha |u|^{z_\alpha} \left[1- i A_\alpha \tan{\left(
\frac{\pi z_\alpha }{2}\right)} u/|u| \right]}$ with explicit real
constants $E_\alpha > 0$, $A_\alpha \in [-1,1]$ given in \cite{Popk16} (L\'evy universality class).
\\[2mm]
$\bullet$ Case 3: If $\alpha \in \mathbb{I}_\alpha$ then $z_{\alpha} =3/2$. (a) 
If there is no diffusive mode $\beta \in \mathbb{I}_\alpha$, then $\hat{f}_{\alpha}(u) = \hat{f}^{MCT}_{KPZ}(u)$ 
given in \cite{Popk16} (KPZ universality class).
b) If there is at least one diffusive mode $\beta \in \mathbb{I}_\alpha$, then 
$\hat{f}_{\alpha}(u) = \hat{f}^{MCT}_{mKPZ}(u)$  given in \cite{Popk16} (modified KPZ universality class).\\[2mm]
(2) The non-linear recursion for the dynamical exponents in case 2 has as unique solution the
sequence of Kepler ratios of Fibonacci numbers $z_i = F_{i+1}/F_i$, starting
from $z_3=3/2$ (if at least one diffusive mode $\beta \in \mathbb{I}_\alpha$), or $z_4=5/3$ 
(if no diffusive mode but at least one KPZ mode $\beta \in \mathbb{I}_\alpha$)
or else the golden mean $z_i = (1+\sqrt{5})/2$ for all modes $i$.
\end{theorem}

\begin{remark}
The subballistic scaling $z_\alpha>1$ is motivated by the locality of interactions, conservation laws 
and currents.
Since all dynamical exponents that can appear are Kepler ratios of neighbouring Fibonacci
numbers we call the whole family of universality classes comprising diffusion, L\'evy, KPZ and
modified KPZ the {\em Fibonacci universality classes}. 
\end{remark}

The main ingredients in the proof of item (a) are strict hyperbolicity and power counting of the 
leading singularities in the Fourier-Laplace representation of the mode-coupling equation. Item (b) 
follows from the recursion of the Fibonacci numbers by a judiciously chosen ordering of the modes
belonging to case 2.

We stress that the theorem deals with the function $S_\alpha(x,t)$ satisfying the mode coupling
equations \eref{mctdiag}. There is no general rigorous result how this function relates to the
true scaling limit of the dynamical structure function $S^{\alpha\alpha}_k(t)$. However,
in the diffusive case 1 one expects the Gaussian scaling function to be the true scaling limit, up to possible
logarithmic corrections. For specific models there are numerical \cite{Popk14a,Popk15a,Popk15b} 
and mathematically rigorous results \cite{Bern14,Bern16} that suggest that the true scaling form
in case 2 is indeed generally a L\'evy distribution. However, the coefficients
$A_\alpha, E_\alpha$ arising from the mode-coupling equations are not believed to correspond to the true values.
The scaling limit of $S_\alpha(x,t)$ has a closed expression
in Fourier-Laplace representation also in case 3. However, for the case of a single conservation
law (the usual KPZ universality class) it is known that this scaling function, studied in detail in
\cite{Cola01,Frey96}, is not exact but rather given by the Pr\"ahofer-Spohn scaling function \cite{Prae04}.
Correspondingly, one does not expect the scaling forms $\hat{f}^{MCT}_{KPZ}(u)$ and $\hat{f}^{MCT}_{mKPZ}(u)$ 
solving the mode coupling equations \eref{mctdiag} to exact either.

With these provisos Theorem 1 shows that, according to the mode coupling theory reviewed above, 
the dynamical universality classes of all modes are fully determined by 
whether or not the diagonal elements of the mode coupling matrices vanish. This in turn is
fully determined by the stationary current-density density relation. Thus, as long as mode-coupling theory is valid, 
one can read off from this macroscopic
stationary quantity alone the dynamical universality classes of all modes.

\section{Two-lane lattice gas for the modified KPZ universality class}

The modified KPZ universality class \cite{Spoh15} arises for
$G^1_{11} = G^1_{22} = 0$ and $G^2_{11} \neq 0, G^2_{22} \neq 0$,  see case 3b) in Theorem 1.
However, so far no microscopic model with this property has been proposed.
Here we present a two-lane lattice gas that belongs to case 3b) in Theorem 1
for all values of the conserved densities without fine-tuning of
model parameters. This model consists of two coupled one-dimensional lattice
gases without self-interaction, but where the jump rates between sites $k$ and
$k+1$ of the particles of one gas depend on the number of particles of the other gas
on the same pair of sites. It is convenient to describe this a model as a two-lane
lattice gas model in the spirit of the multi-lane exclusion processes of Popkov and
Salerno \cite{Popk04}, but without requiring exclusion. 

We denote the number of particles on site $k$ on lane $i$ as particles of type by $n_k^i$. Thus the
local state variable is the pair $\eta_k = (n_k^1,n_k^2)\in \N_0^2$. We introduce the parameters
\be 
r_i = \half(w_i + f_i), \quad \ell_i = \half(w_i - f_i) , \quad
g^\pm = \half(\alpha \pm \gamma)
\ee
with strictly positive constants $w_i,\alpha > 0$, and $|f_i| \leq w_i$, $|\gamma| \leq \alpha$. 
We also define the mean pair occupation numbers
\be 
\bar{n}^i_{k} = \half \left( n^i_k + n^i_{k+1}\right).
\ee
Informally, the stochastic dynamics is then defined as follows. A particle on lane $1$ jumps independently of 
all other particles on lane $1$ after an exponential waiting time from site $k$ to site $k+1$ of lane $1$ 
with rate $r^1(\bar{n}^2_{k})=r_1 + g^+\bar{n}^2_{k}$ and from site $k+1$ to $k$ 
with rate $\ell^1(\bar{n}^2_{k})=\ell_1 + g^- \bar{n}^2_{k}$. Likewise, particles on lane $2$ jump 
with rates $r^2(\bar{n}^1_{k})=r_2 + g^+ \bar{n}^1_{k}$ and from site $k+1$ to $k$ and
with rate $\ell^2(\bar{n}^1_{k})=\ell_2 + g^-\bar{n}^1_{k}$. Thus $w_i$ are the ``bare'' jump rates (i.e., in the absence of
interaction), $f_i$ are the bare jump biases, $\alpha$ is the interaction strength and
$\gamma$ is the interaction asymmetry.

With the updated state variable
\be 
\eta^{i;k,k'}_l = \left\{ \ba{ll} n_k^i -1 & \mbox{if } l=k \\ n_{k'}^i +1 & \mbox{if } l=k' \\
n_k^i  & \mbox{else} \ea \right.
\ee 
the generators for independent random walkers read
\be 
\mathcal{L}_i = \sum_k \left[ r_i \left( f(\eta^{i;k,k+1}) - f(\eta)\right) + \ell_i \left( f(\eta^{i;k+1,k}) - f(\eta)\right)\right].
\ee
The interaction between the lanes is given by the generator
\bea 
\mathcal{L}_{I} \left(f(\eta)\right) & = &  \sum_k \left\{\bar{n}^2_{k}  \left[ g^+ f(\eta^{1;k,k+1}) 
+ g^- f(\eta^{1;k+1,k}) - (g^+ + g^-) f(\eta) \right] \right. \nonumber \\
& & + \left. \bar{n}^1_{k}  \left[ g^+ f(\eta^{2;k,k+1}) 
+ g^- f(\eta^{2;k+1,k}) - (g^+ + g^-) f(\eta) \right]\right\}.
\eea
The generator of the full interacting process is then 
\bel{LmodKPZ}
\mathcal{L} = \mathcal{L}_1 + \mathcal{L}_2 + \mathcal{L}_{I}.
\ee

The interaction between the lanes does not change the invariant measure of the
non-interacting part. Thus one arrives at
\begin{proposition}
\label{Prop:statmeas}
For parameters $\rho_{1,2} \geq 0$ the product measure with factorized Poisson marginals
\bel{pm}
\mu(\eta_k) = \frac{\rho_1^{n^1_k}\rme^{-\rho_1}}{(n^1_k)!}  \frac{\rho_1^{n^2_k}\rme^{-\rho_2}}{(n^2_k)!} 
\ee
for the occupation at site $k$ is a translation invariant measure of the process defined by the generator
\eref{LmodKPZ}.
\end{proposition}

\begin{proof}
We prove the proposition for the finite torus $\Lambda=\{1,\dots,L\}$ with $L$ sites in the quantum operator 
formalism \cite{Schu01,Sudb95}. 
Since \eref{pm} is invariant for the non-interacting part $\mathcal{L}_1 + \mathcal{L}_2$ of the
generator one needs to prove only invariance under $\mathcal{L}_{I}$. For 
$n\in\N_0$ let $\ket{n}$ be the infinite-dimensional
vector with components $(\ket{n})_i=\delta_i$ and define  the tensor vectors $\ket{n^1,n^2}=\ket{n^1}\otimes\ket{n^2}$.
Furthermore, define the matrices $\mathbf{1}$, $a^{\pm}$ and $\hat{n}$ by
$a^{+}\ket{n} = \ket{n+1}$, $a^-\ket{n} = n \ket{n-1}$, $\hat{n}\ket{n} = n \ket{n}$, $\mathbf{1}\ket{n} = n \ket{n}$
and also the tensor products $a^{1,\pm} = a^\pm \otimes \mathbf{1}$, $a^{2,\pm} = \mathbf{1} \otimes a^\pm$,
$\hat{n}^1= \hat{n} \otimes \mathbf{1}$, $\hat{n}^2 = \mathbf{1} \otimes \hat{n}$, $\mathbf{1}_{12} 
= \mathbf{1} \otimes \mathbf{1}$ and $X_k = \mathbf{1}_{12}^{\otimes(k-1)} \otimes X  \otimes  
\mathbf{1}_{12}^{\otimes(L-k)}$ where
$X$ is any of the two-fold tensor products just defined. The matrix form $H_I$ of the generator $\mathcal{L}_{I}$ is then given by
$ H_I = \sum_k (g_k^1 + g_k^2)$
with
\bea 
g_k^1 & = &
- \half  \left\{ \left(\hat{n}^{2}_k+\hat{n}^{2}_{k+1}\right) 
\left[ g^+ \left( a^{1,-}_k a^{1,+}_{k+1} - \hat{n}^{1}_k \right) 
+ g^- \left( a^{1,+}_k a^{1,-}_{k+1} - \hat{n}^{1}_{k+1} \right)\right] \right\}
\\
g_k^2 & = & 
- \half \left\{ \left(\hat{n}^{1}_k+\hat{n}^{1}_{k+1}\right) 
\left[ g^+ \left( a^{2,-}_k a^{2,+}_{k+1} - \hat{n}^{2}_k \right) 
+ g^- \left( a^{2,+}_k a^{2,-}_{k+1} - \hat{n}^{2}_{k+1} \right)\right] \right\}.
\eea
Let $\ket{\rho} = \rme^{-\rho} \sum_{n=0}^\infty \rho^n/(n!) \ket{n}$ and
$\ket{\rho_1,\rho_2} =( \ket{\rho_1} \otimes \ket{\rho_2} )^{\otimes L}$ for $\rho_{1,2}\geq 0$.
One has $a^{\alpha,+}_k \ket{\rho_1,\rho_2} = \hat{n}^\alpha_k / \rho_\alpha$ and
$a^{\alpha,-}_k \ket{\rho_1,\rho_2} = \rho_\alpha \ket{\rho_1,\rho_2}$ \cite{Schu01}. It follows that
\bea 
& &\left( a^{\alpha,-}_k a^{\alpha,+}_{k+1} - \hat{n}^{\alpha}_k \right) \ket{\rho_1,\rho_2} 
= \left(\hat{n}^{\alpha}_{k+1}  - \hat{n}^{\alpha}_k \right)\ket{\rho_1,\rho_2} \\
& & \left( a^{\alpha,+}_k a^{\alpha,-}_{k+1} - \hat{n}^{\alpha}_{k+1} \right) \ket{\rho_1,\rho_2} 
= \left(\hat{n}^{\alpha}_{k}  - \hat{n}^{\alpha}_{k+1} \right)\ket{\rho_1,\rho_2}
\eea
and therefore
\bea 
& & g_k^1\ket{\rho_1,\rho_2}  = - \frac{\gamma}{2} \left(\hat{n}^{2}_k+\hat{n}^{2}_{k+1}\right)
\left(\hat{n}^{1}_{k+1}  - \hat{n}^{1}_k \right)\ket{\rho_1,\rho_2} \\
& & g_k^2\ket{\rho_1,\rho_2}  = - \frac{\gamma}{2} \left(\hat{n}^{1}_k+\hat{n}^{1}_{k+1}\right)
\left(\hat{n}^{2}_{k+1}  - \hat{n}^{2}_k \right)\ket{\rho_1,\rho_2}.
\eea
The telescopic property of the lattice sum then yields $H_I \ket{\rho_1,\rho_2} = 0$. \hspace*{15mm}\qed
\end{proof}

In Proposition \ref{Prop:statmeas} the parameters $\rho_{i}=\exval{n^i_k}$ are the conserved stationary 
densities and one finds immediately the
compressibility matrix $\mathbf{K}$ with matrix elements $K_{\alpha\beta} = \rho_\alpha \delta_{\alpha\beta}$.
From the definition \eref{LmodKPZ} of the generator one computes the microscopic 
currents defined up to an irrelevant constant by the discrete continuity equation \eref{dce}. The
factorization of the invariant measure then yields the stationary current-density relation
\bel{scd} 
j^1(\rho_1,\rho_2) = \rho_1(f_1 +\gamma \rho_2), \quad
j^2(\rho_1,\rho_2) = \rho_2(f_2 +\gamma \rho_1).
\ee
Remarkably, the stationary currents depend only on the hopping biases $f_{1,2}$ and the
interaction asymmetry $\gamma$, not on the interaction strength $\alpha$.\footnote{The product measure
\eref{pm} remains invariant also for different interaction strength $\alpha_1\neq\alpha_2$ which leaves the
currents unchanged.
However, equal interaction asymmetry is required.}

The main result is the following.

\begin{theorem}
\label{Theo:GmodKPZ}
Define the quantities
\bea 
\Delta & := & \sqrt{\left(f_1 -f_2 +\gamma (\rho_2-\rho_1)\right)^2 + 4 \gamma^2 \rho_1 \rho_2} \\
\xi & := & \frac{\Delta - \left(f_2 - f_1 + \gamma(\rho_1-\rho_2)\right)}{2 \gamma \sqrt{\rho_1 \rho_2}} , \quad
y := \frac{f_2 - f_1 + \gamma(\rho_1-\rho_2)}{2 \gamma \sqrt{\rho_1 \rho_2}}.
\eea
For all bare hopping rates $w_{1,2},f_{1,2}$, all strictly positive densities $\rho_{1,2}$,
and all non-zero interaction parameters $\alpha,\gamma$,
the current Jacobian of the process defined by \eref{LmodKPZ} with invariant measure given in Proposition 
\ref{Prop:statmeas} is non-degenerate and has eigenvalues
\bel{collvecmodKPZ} 
v_\pm = \half \left(f_1 +f_2 + \gamma (\rho_1 + \rho_2) \pm \Delta\right).
\ee
The mode coupling matrices \eref{mcm} where mode 1 (2) has collective velocity $v_+$ ($v_-$) 
are given by
\bel{GmodKPZ}
\mathbf{G}^\alpha = - g_\alpha \left(\ba{cc} -1 & y \\ y & 1 \ea \right)  
\ee
with
\bea 
g_1 & = & \rho_1 \sqrt{\frac{\gamma^3 \xi}{\Delta^3}}  
(\gamma(\rho_1+\rho_2) + \Delta + f_2 - f_1) \\
g_2 & = & \rho_2 \sqrt{\frac{\gamma^3 \xi}{\Delta^3 }}
\left(\gamma(\rho_1+\rho_2) - \left(\Delta + f_2 - f_1\right)\right) .
\eea
\end{theorem}

\begin{remark}
For $\gamma = 0$ one has $g_1=g_2=0$ for all
$f_1,f_2,\rho_1,\rho_2$ so that both modes are diffusive with drift velocites $f_{1,2}$.
For interaction asymmetry $\gamma > 0$ and strictly positive densities the drift velocities 
of the two modes are different even for equal individual bare hopping asymmetries $f_1=f_2$.
\end{remark}

\begin{remark}
For $\gamma > 0$ and strictly positive densities one has $g_1 \neq 0$ for all $f_1,f_2$.
On the other hand, $g_2 \neq 0$ if and only if $f_1 \neq f_2$. Thus according to case 3 of 
Theorem \ref{Theo:Fibonacci} one expects that
for any $\gamma > 0$ and strictly positive densities $\rho_{1,2}$ 
both modes are KPZ if $f_1\neq f_2$ whereas for equal asymmetries $f_1= f_2$ mode 1
is modified KPZ and mode 2 is diffusive, without fine-tuning of parameters.
\end{remark}

\begin{remark}
The offdiagonal elements of the mode coupling matrices vanish for $j^1_1=j^2_2$, which is equivalent
to $f_1-f_2 = \gamma (\rho_1-\rho_2)$.
\end{remark}

\begin{proof}
The proof of Theorem \ref{Theo:GmodKPZ} is computational. From the current-density relation \eref{scd}
one obtains the current Jacobian 
\be 
\mathbf{J} = \left(\ba{cc} f_1 +\gamma \rho_2 & \gamma \rho_1 \\
\gamma \rho_2 & f_2 +\gamma \rho_1 \ea \right).
\ee
Solving the eigenvalue equation proves \eref{collvecmodKPZ}. With
\be 
u = \sqrt{\frac{\rho_1}{\rho_2}}
\ee
the diagonalizing matrix defined by 
\bel{Jdiag} 
\mathbf{V} := \mathbf{R}\mathbf{J}\mathbf{R}^{-1} =  \left( \ba{cc} v_+ & 0 \\
0 & v_- \ea \right) .
\ee
is computed to be
\bel{Rtrafoboson} 
\mathbf{R} = \left( \ba{cc} x_+ & x_+ \xi^{-1} u \\
- x_- \xi^{-1} u^{-1} & x_- \ea \right)
\ee
with $\det(\mathbf{R}) = x_+x_-(1+\xi^{-2})$ and with free parameters $x_\pm$. 

The parameters $x_\pm$ are fixed by the normalization condition \eref{modenorm}. One finds
\be 
x_+ = \frac{1}{\sqrt{\rho_1 (1+\xi^{-2})}} 
= \sqrt{\frac{\gamma \xi}{\delta u}}, \quad 
x_- = \frac{1}{\sqrt{\rho_2 (1+\xi^{-2})}} 
= \sqrt{\frac{\gamma \xi u}{\delta}}
\ee

For the Hessians one finds
\be 
\mathbf{H}^1 = \mathbf{H}^2 = \gamma \left( \ba{cc} 0 & 1 \\ 1 & 0 \ea \right)=: H
\ee
and therefore
\be 
(\mathbf{R}^{-1})^T \mathbf{H} \mathbf{R}^{-1} =  2 \gamma \rho_1\rho_2 \frac{\gamma}{\delta}  
\left( \ba{cc} 
 1 & - y \\  - y & - 1
\ea \right).
\ee
The definition \eref{mcm} then yields \eref{GmodKPZ}. \hspace*{59mm} \qed
\end{proof}

\section{Criterion for L\'evy universality classes for systems with two conservation laws}

In the case of two conservation  laws we denote the universality classes of the two modes by
a pair $(\cdot,\cdot)$ where the possible entries are $D$ for diffusion ($z=2$),
$\frac{3}{2}L$ for the L\'evy universality class with $z=3/2$ and $GM$ for the L\'evy universality 
class with the golden mean $z=(1+\sqrt{5})/2$.

\subsection{Diagonalization of the current Jacobian}

Let
\bel{J}
\mathbf{J} = \left( \ba{cc} j^1_1 & j^1_2 \\ j^2_1 & j^2_2 \ea \right)
\ee
be a current Jacobian. The two eigenvalues of $\mathbf{J}$ are
\be 
v_\pm = \frac{1}{2} \left( j^1_1 + j^2_2 \pm \Delta \right)
\ee
with
\be 
\Delta = \sqrt{( j^1_1 - j^2_2)^2 + 4 j^1_2 j^2_1}.
\ee
We consider only the strictly hyperbolic case $\Delta > 0$.
The diagonalizer $R$ with the property \eref{Jdiag} reads
\bel{R2} 
\mathbf{R} = \left( \ba{cc} x_+  & x_+ \frac{2j^1_2}{\Delta + (j^1_1 - j^2_2 )} \\
- x_-  \frac{2j^2_1}{\Delta + (j^1_1 - j^2_2 )} & x_- \ea \right)
\ee
with constants $x_\pm$ satisfying $x_+x_-\neq 0$ and to be chosen such that
$\mathbf{R}$ has well-defined limits $j^1_2\to 0$ or $j^2_1\to 0$.

\subsection{Non-KPZ universality classes}

According to cases 1 and 2 in Theorem \ref{Theo:Fibonacci} mode coupling theory predicts two non-KPZ
universality classes for $G^1_{11} = G^2_{22} = 0$ and specifically 
\begin{itemize}
\item $(DD)$ if and only if  $G^1_{22} = G^2_{11} = 0$ 
\item $(D,\frac{3}{2}L)$ if and only if  $G^1_{22} = 0$, $G^2_{11} \neq 0$ 
\item $(\frac{3}{2}L,D)$ if and only if  $G^1_{22} \neq 0$, $G^2_{11} = 0$ 
\item $(GM,GM)$ if and only if  $G^1_{22} \neq 0$, $G^2_{11} = 0$ .
\end{itemize}
The diagonal elements of the mode coupling matrices have been computed
explicitly in \cite{Popk15a}, albeit in a form that does not directly express them
in terms of the current-density relation. Moreover, the expressions in \cite{Popk15a}
depend on the normalization factors $x_\pm$ in \eref{R2} fixed by \eref{modenorm}, which, 
however, is irrelevant with regard to whether a diagonal element
is zero or not and therefore irrelevant to the question which universality class
one expects. A more ``user-friendly'' form 
that expresses the conditions on the various allowed universality classes directly in terms of the current derivatives 
is the
following result.

\begin{theorem}
Let $\mathbf{J}$ be a current Jacobian and $\mathbf{G}^\alpha$ be mode coupling matrices
as defined in \eref{mcm}. Then one has the generic non-KPZ conditions $G^1_{22} = G^2_{11} = 0$ if and only if 
\bea 
\label{nonKPZ1}
j^2_1 \left(2  j^1_{12} +  j^2_{22}\right) + 
j^1_2 j^2_{11} - j^1_{11} \left(j^2_2-j^1_1\right) & = & 0 \\
\label{nonKPZ2}
j^1_2\left(2  j^2_{12} +  j^1_{11}\right) + 
j^2_1 j^1_{22} + j^2_{22} \left(j^2_2-j^1_1\right) & = & 0 
\eea
and the specific conditions
\bea 
\label{DD}
\mbox{(D,D)} & \Leftrightarrow & j^2_1 j^1_{22} + j^1_2 j^1_{11}  = j^2_1 j^2_{22} + j^1_2 j^2_{11} = 0 
\eea
for two diffusive modes,
\bea
\label{DL}
\mbox{(D,$\frac{3}{2}$L)} & \Leftrightarrow & \left(j^2_1\right)^2 j^1_{22} + j^1_2 j^2_1 j^1_{11}  =  
\half \left(j^1_1 - j^2_2 - \delta\right)
 \left( j^2_1 j^2_{22} + j^1_2 j^2_{11}\right) \\
 \label{LD}
\mbox{($\frac{3}{2}$L,D)} & \Leftrightarrow & \left(j^2_1\right)^2 j^1_{22} + j^1_2 j^2_1 j^1_{11}  =  
\half \left(j^1_1 - j^2_2 + \delta\right)
 \left( j^2_1 j^2_{22} + j^1_2 j^2_{11}\right) 
 \eea
 for the mixed case with one diffusive and one 3/2-L\'evy mode, and
 \bea
 \label{GMGM}
\mbox{(GM,GM)} & \Leftrightarrow & \left(j^2_1\right)^2 j^1_{22} + j^1_2 j^2_1 j^1_{11}  \neq  
\half \left(j^1_1 - j^2_2 \pm \delta\right)
 \left( j^2_1 j^2_{22} + j^1_2 j^2_{11}\right)  
\eea
for two golden mean L\'evy modes.
\end{theorem}

\begin{proof}
We invert \eref{mcm} to find $2 R^T G^\alpha R = R_{\alpha1} H^1 + R_{\alpha2} H^2$.
Requiring the the generic non-KPZ conditions $G^1_{22} = G^2_{11} = 0$ and using that the 
mode coupling matrices and the
Hessians are symmetric leads to six independent equations involving the
current derivatives.
In term of the parameters
\be 
 u = \sqrt{\frac{j^1_2}{j^2_1}} , \quad
\xi = \frac{\Delta - \left(j^2_2 - j^1_1\right)}{2 \sqrt{j^1_2 j^2_1}}, \quad
y = \frac{j^2_2 - j^1_1}{2 \sqrt{j^1_2 j^2_1}}
\ee
they read
\bea 
\xi  u^{-1} j^1_{11} + j^2_{11} & = &
2 G^1_{22} \frac{(x_-)^2}{x_+ u \xi^{-1}} \left(\frac{ u^{-1}}{\xi}\right)^2 -
4 G^1_{12} x_-  u^{-2} \\
\xi  u^{-1} j^1_{12} + j^2_{12} & = &
- 2 G^1_{22} \frac{(x_-)^2}{x_+ u \xi^{-1}} \frac{ u^{-1}}{\xi} -
4 G^1_{12} x_-  u^{-1} \frac{1-\xi^2}{2\xi} \\
\xi  u^{-1} j^1_{22} + j^2_{22} & = &
2 G^1_{22} \frac{(x_-)^2}{x_+ u \xi^{-1}} + 4 G^1_{12} x_- \\
\frac{ u^{-1}}{\xi} j^1_{11} - j^2_{11}
& = & - 2 G^2_{11} \frac{\left(x_+ u \xi^{-1}\right)^2}{x_-}
\left(\xi  u^{-1}\right)^2 
+ 4 G^2_{12} x_+ u \xi^{-1}  u^{-2} \\
\frac{ u^{-1}}{\xi} j^1_{12} - j^2_{12}
& = & - 2 G^2_{11} \frac{\left(x_+ u \xi^{-1}\right)^2}{x_-}
\xi  u^{-1} 
+ 4 G^2_{12} x_+ u \xi^{-1}  u^{-1} \frac{1-\xi^2}{2\xi} \\
\frac{ u^{-1}}{\xi} j^1_{22} - j^2_{22}
& = & - 2 G^2_{11} \frac{\left(x_+ u \xi^{-1}\right)^2}{x_-} - 4 G^2_{12} x_+ u \xi^{-1}.
\eea

Since $G^1_{12}$ and $G^2_{12}$ are arbitrary, we can introduce arbitrary new constants
\be 
A = \frac{- 4 G^1_{12} x_- + 4 G^2_{12} x_+ u \xi^{-1}}{\xi^{-1} + \xi}, \quad
B = \frac{- 4 G^1_{12} x_- \xi^{-1}  - 4 G^2_{12} x_+ u \xi^{-1} \xi}{\xi^{-1} + \xi}
\ee
so that the six non-KPZ equations become
\bea 
j^1_{11} & = &
 u^{-1} \left( \frac{2 G^1_{22} \frac{(x_-)^2}{x_+ u \xi^{-1}} \xi^{-2} - 2 G^2_{11} \frac{\left(x_+ u \xi^{-1}\right)^2}{x_-} \xi^2}{\xi^{-1} + \xi}
 + A \right) \\
j^1_{12} & = &
\frac{- 2 G^1_{22} \frac{(x_-)^2}{x_+ u \xi^{-1}} \xi^{-1} - 2 G^2_{11} \frac{\left(x_+ u \xi^{-1}\right)^2}{x_-} \xi}{\xi^{-1} + \xi}  
+ A \frac{1-\xi^2}{2\xi}\\
j^1_{22} & = &
 u  \left( \frac{2 G^1_{22} \frac{(x_-)^2}{x_+ u \xi^{-1}} - 2 G^2_{11} \frac{\left(x_+ u \xi^{-1}\right)^2}{x_-}}{\xi^{-1} + \xi} - A \right) \\
j^2_{11} & = &  u^{-2} 
\left( \frac{2 G^1_{22} \frac{(x_-)^2}{x_+ u \xi^{-1}} \xi^{-3} + 2 G^2_{11} \frac{\left(x_+ u \xi^{-1}\right)^2}{x_-} \xi^3}{\xi + \xi^{-1}}
+ B \right)\\
j^2_{12} & = &  u^{-1} 
\left(- \frac{2 G^1_{22} \frac{(x_-)^2}{x_+ u \xi^{-1}} \xi^{-2} - 2 G^2_{11} \frac{\left(x_+ u \xi^{-1}\right)^2}{x_-} \xi^2}{\xi + \xi^{-1}} 
+ B \frac{1-\xi^2}{2\xi} \right)\\
j^2_{22} & = & 
\frac{2 G^1_{22} \frac{(x_-)^2}{x_+ u \xi^{-1}} \xi^{-1} + 2 G^2_{11} \frac{\left(x_+ u \xi^{-1}\right)^2}{x_-} \xi}{\xi + \xi^{-1}}
- B.
\eea

Now we use the fact that 
$\xi^{\pm 1} = \sqrt{1+y^2} \mp y$
to write
\be 
1 = \xi^{\pm 2} \pm 2y \xi^{\pm 1}, \quad
\xi^{\pm 3} = \xi^{\pm 1} \mp 2y \xi^{\pm 2}.
\ee
Defining
\be 
C = \frac{2\left(\frac{G^1_{22} (x_-)^2}{x_+ u \xi} - \frac{G^2_{11} \left(x_+ u \right)^2}{x_-}\right)}{\xi+\xi^{-1}}, \quad
D = \frac{2\left(\frac{G^1_{22} (x_-)^2}{x_+ u } + \frac{G^2_{11} \left(x_+ u \right)^2}{x_- \xi} \right)}{\xi+\xi^{-1}}
\ee
this yields
\bea 
\frac{2 G^1_{22} \frac{(x_-)^2}{x_+ u \xi^{-1}} - 2 G^2_{11} \frac{\left(x_+ u \xi^{-1}\right)^2}{x_-}}{\xi+\xi^{-1}} & = & C - 2y D \\
\frac{2 G^1_{22} \frac{(x_-)^2}{x_+ u \xi^{-1}} \xi^{-3} + 2 G^2_{11} \frac{\left(x_+ u \xi^{-1}\right)^2}{x_-} \xi^3}{\xi+\xi^{-1}} & = &
D + 2y C .
\eea
Thus the six non-KPZ conditions can be recast in the form
\bea 
j^1_{11} & = &
 u^{-1} \left( C
 + A \right) \\
j^1_{12} & = &
- D  
+ y A \\
j^1_{22} & = &
 u  \left( C - 2y D - A \right) \\
j^2_{11} & = &  u^{-2} 
\left( D + 2y C
+ B \right)\\
j^2_{12} & = &  u^{-1} 
\left(- C 
+ y B \right)\\
j^2_{22} & = & 
D - B
\eea

Next we choose the arbitrary functions as
\be 
A =  u  j^1_{11} - C, \quad B = D - j^2_{22}
\ee
to obtain
\bea 
j^1_{12} & = & - (D  + y C) + y  u  j^1_{11} \\
j^1_{22} & = & 2  u  (C - y D) -  u^{2} j^1_{11} \\
j^2_{11} & = &  
2  u^{-2} (D + y C) -  u^{-2} j^2_{22} \\
j^2_{12} & = & 
-  u^{-1} (C - y D) - y  u^{-1} j^2_{22}.
\eea
With the short-hands
\be 
F_1 = 2  u  (C - y D), \quad F_2 = 2  (D + y C)
\ee
these equations take the form
\bea 
2  j^1_{12} +  j^2_{22} + u^2 j^2_{11} - 2 u y j^1_{11} & = & 0 \\
2  j^2_{12} +  j^1_{11} + u^{-2} j^1_{22} + 2  u^{-1} y j^2_{22} & = &  0 \\
j^1_{22} + u^2 j^1_{11}  & = &  F_1 \\
j^2_{22} + u^2 j^2_{11}  & = &  F_2
\eea

Next we observe
\bea
F_1 
& = & u \left( 2 G^1_{22} \frac{(x_-)^2}{x_+ u \xi^{-1}} \xi^{-1} - 2 G^2_{11} \frac{\left(x_+ u \xi^{-1}\right)^2}{x_-} \xi \right) \\
F_2 
& = & 2 G^1_{22} \frac{(x_-)^2}{x_+ u \xi^{-1}} \xi^{-2} + 2 G^2_{11} \frac{\left(x_+ u \xi^{-1}\right)^2}{x_-} \xi^2 .
\eea
Thus,  by setting the respective diagonal elements $G^\alpha_{\beta\beta}$ to zero,
\bea 
& & (D,D): G^1_{22} = G^2_{11} = 0 \Rightarrow 
F_1 = F_2 = 0 \\
& & (D,\frac{3}{2}L): G^1_{22} = 0, G^2_{11} \neq 0 \Rightarrow 
F_1 = - \xi^{-1} u F_2 \neq 0 \\
& & (\frac{3}{2}L,D): G^1_{22} \neq 0 , G^2_{11} = 0 \Rightarrow 
F_1 = \xi u F_2 \neq 0  \\
& & (GM,GM): G^1_{22} \neq 0 , G^2_{11} \neq 0 \Rightarrow 
F_1 \neq \pm u \xi^{\pm 1} F_2 .
\eea

In terms of the derivatives one has
\bea 
& & uy = \frac{j^2_2-j^1_1}{2j^2_1}, \quad 
u^{-1}y = \frac{j^2_2-j^1_1}{2j^1_2} \\
& & u \xi = \frac{\delta - \left(j^2_2 - j^1_1\right)}{2j^2_1}, \quad
u \xi^{-1} = \frac{\delta + \left(j^2_2 - j^1_1\right)}{2j^2_1}
\eea
which yields the conditions \eref{nonKPZ1} - \eref{GMGM} as stated in the theorem. Conversely, one proves
that the required diagonal elements vanish by assuming the conditions to be valid and using the definition \eref{mcm}
of the mode coupling matrices. \hspace*{20mm} \qed
\end{proof}

\end{document}